\documentclass[11pt]{article}

\usepackage{cite}
\usepackage{amssymb}
\usepackage{amsmath}
\usepackage{amsfonts}
\usepackage{amsthm}
\usepackage[all]{xy}
\usepackage{enumerate}
\usepackage{colonequals}
\usepackage{enumitem}

\usepackage{color}

\usepackage{url}

\usepackage[mathscr]{euscript}

\newtheorem{theorem}{Theorem}[section]

\newtheorem*{theorem*}{Theorem}
\newtheorem{thm}{Theorem}[section]

\newtheorem{definition}{Definition}[section]

\newtheorem{lemma}{Lemma}[section]
\newtheorem{claim}{Claim}[section]

\numberwithin{equation}{section}

\allowdisplaybreaks[2]

\newcommand{\rest}{\upharpoonright}

               \def\P{{\mathbb{P}}}  \def\R{{\mathbb{R}}}        
  
\newcommand{\G}{\Gamma}

\renewcommand\a{\alpha}

\newcommand\s{\sigma}

\renewcommand\l{\lambda}

  \newcommand\Sch{\operatorname{Schreier}}

\allowdisplaybreaks

\begin{document}

\title{Flexible stability and nonsoficity}
\author{Lewis Bowen and Peter Burton}

\maketitle

\begin{abstract}
\noindent A sofic group $G$ is said to be flexibly stable if every sofic approximation to $G$ can converted to a sequence of disjoint unions of Schreier graphs by modifying an asymptotically vanishing proportion of edges. We establish that if $\mathrm{PSL}_d(\mathbb{Z})$ is flexibly stable for some $d \geq 5$ then there exists a group which is not sofic. \end{abstract}

\tableofcontents

\section{Introduction}

\subsection{Sofic groups}

Soficity is a finite approximation property for countable discrete groups which has received considerable attention in recent years. A group is called sofic if it admits a sofic approximation, which is a sequence of partial actions on finite sets that asymptotically approximates the action of the group on itself by left-translations. The precise definition appears below. Soficity can be thought of as a common generalization of amenability and residual finiteness. We refer the reader to \cite{pestov-sofic-survey, capraro-lupini} for surveys.\\
\\
It is a famous open problem to determine whether every countable discrete group is sofic. It is also widely open to classify sofic approximations to well-known groups, for example by showing that every sofic approximation is asymptotically equivalent to an approximation by actions on finite sets (as opposed to partial actions). If a group has this latter property, it is called flexibly stable. The main result of this paper is that if $\mathrm{PSL}_d(\mathbb{Z})$ is flexibly stable for some $d \geq 5$ then there is a nonsofic group. The proof gives an explicit group $G$, constructed as a quotient of an HNN-extension of $\mathrm{PSL}_d(\mathbb{Z})$, that is not sofic if $\mathrm{PSL}_d(\mathbb{Z})$ is flexibly stable.\\
\\
We now formulate precise definitions to state the result.
\begin{definition} Let $G$ be a countable discrete group. A \textbf{sofic approximation} to $G$ consists of a sequence $(V_n)_{n=1}^\infty$ of finite sets and a sequence $(\sigma_n)_{n=1}^\infty$ of functions $\sigma_n:G \to \mathrm{Sym}(V_n)$ such that the following conditions hold, where we write $\sigma_n^g$ instead of $\sigma_n(g)$. \begin{itemize} \item \underline{Asymptotic homomorphisms}: For every fixed pair $g,h \in G$ we have \[ \lim_{n \to \infty} \frac{1}{|V_n|} |\{ v \in V_n: \sigma_n^g(\sigma_n^h(v))= \sigma_n^{gh}(v) \} | = 1. \] \item \underline{Asymptotic freeness}: For every fixed nontrivial element $g \in G$ we have \[ \label{eq.asyinj} \lim_{n \to \infty} \frac{1}{|V_n|} |\{v \in V_n: \sigma_n^g(v) = v \}| = 0. \] \end{itemize} We say that $G$ is \textbf{sofic} if there exists a sofic approximation to $G$. \end{definition}

\subsection{Flexible stability}

\begin{definition} A sofic approximation $(\sigma_n:G \to \mathrm{Sym}(V_n))_{n=1}^\infty$ is \textbf{perfect} if each $\sigma_n$ is a genuine group homomorphism. \end{definition}

If $S$ is a finite generating set for $G$ we can endow $V_n$ with the structure of a $S$-labelled directed graph by putting an $s$-labelled edge from $v$ to $\sigma^s(v)$ for each $s \in S$ and $v \in V_n$. Accordingly, we refer to the $V_n$ as the vertex sets of the sofic approximation. With this structure, each connected component of a perfect sofic approximation to $G$ is a Schreier graph on the cosets of a finite-index subgroup of $G$.

\begin{definition}\label{def.conjugacy} Let $\Sigma = (\sigma_n: G \to \mathrm{Sym}(V_n))_{n=1}^\infty$ and $\Xi = (\xi_n: G \to \mathrm{Sym}(V_n))_{n=1}^\infty$ be two sofic approximations to $G$ with the same vertex sets. We say that $\Sigma$ and $\Xi$ are at \textbf{edit-distance zero} if  for each fixed $g \in G$ we have \[ \lim_{n \to \infty} \frac{1}{|V_n|} |\{v \in V_n: \sigma^g(v) = \xi^g(v) \}| = 1.
\]
Now suppose the vertex sets of $\Xi = (\xi_n: G \to \mathrm{Sym}(W_n))_{n=1}^\infty$ are not necessarily the same as the vertex sets of $\Sigma$. We say that $\Sigma$ and $\Xi$ are \textbf{conjugate} if there exist finite sets $U_n$ and injections $\pi_n:V_n \to U_n$, $\rho_n:W_n \to U_n$ such that
$$1 = \lim_{n\to\infty} \frac{|V_n|}{|U_n|}  = \lim_{n\to\infty} \frac{|W_n|}{|U_n|}$$
and such that the sofic approximations $( \pi_{n*}\sigma_n)_{n=1}^\infty$ and $(\rho_{n*}\xi_n)_{n=1}^\infty$ are at edit-distance zero. Here $\pi_{n*}\sigma_n:G \to \mathrm{Sym}(U_n)$ is the map defined by:
$$(\pi_{n*}\sigma_n)^g( \pi_n(v)) = \pi_n( \sigma_n^g(v))$$
for $v \in V_n$ and 
$$(\pi_{n*}\sigma_n)^g(u) = u$$
if $u \in U_n \setminus \pi_n(V_n)$. The map $\rho_{n*}\xi_n:G \to \mathrm{Sym}(U_n)$ is defined similarly.
 \end{definition}

\begin{definition} We say that a sofic group $G$ is \textbf{flexibly stable} if every sofic approximation to $G$ is conjugate to a perfect sofic approximation to $G$. \end{definition}

It is clear that a flexibly stable group is residually finite. It is also clear that free groups are flexibly stable. In \cite{2019arXiv190107182L} it is shown that surface groups are flexibly stable. A group $G$ is said to be \textbf{strictly stable} if every sofic approximation is conjugate to a perfect sofic approximation where the conjugacies $\pi$ and $\rho$ as in Definition \ref{def.conjugacy} are bijections. In \cite{MR3350728} it is shown that finitely generated abelian groups are strictly stable. In \cite{2018arXiv180108381B} it is shown that polycyclic groups are strictly stable. In \cite{2018arXiv180900632B} it is shown that no infinite property (T) group is strictly stable. The most elementary example for which flexible stability is unknown seems to be the direct product of the rank two free group with $\mathbb{Z}$. \\
\\
The main result of this paper is the following.

\begin{theorem} \label{thm} Suppose that $\mathrm{PSL}_d(\mathbb{Z})$ is flexibly stable for some $d \geq 5$. Then there exists a group which is not sofic. \end{theorem}

The nonsofic group of the theorem has the following form. Let $H$ be a countable discrete group with subgroups $A$ and $B$ and suppose there is an isomorphism $\phi:A \to B$. The {\bf HNN extension} $H \ast_\phi$ is defined to be $(H \ast \langle t \rangle)/ N$ where $H \ast \langle t\rangle$ is the free product of $H$ with a copy of $\mathbb{Z}$ and $N$ is the smallest normal subgroup of $H \ast \langle t \rangle$ containing all elements of the form $t a t^{-1} \phi(a)^{-1}$ for $a \in A$. We will need a $\mathrm{mod} \,\,2$ version of the construction above. So let $N_2$ be the smallest normal subgroup of $H \ast \langle t \rangle$ containing all elements of the form $tat^{-1}\phi(a)^{-1}$ for $a \in A$ along with $t^2$. Let $H \ast_\phi/2 = (H \ast \langle t \rangle)/N_2$. \\
\\
In Section \ref{sec:general}, we show that if $H$ is flexibly stable and if $H,A,B$ and $\phi$ satisfy certain technical conditions then the group $H \ast_\phi/2$ cannot be sofic. This part of the argument is completely general in that it does not use anything specific to $\mathrm{PSL}_d(\mathbb{Z})$. The rest of the paper involves constructing two subgroups $A$ and $B$ of $\mathrm{PSL}_d(\mathbb{Z})$ and showing that they possess the required properties. This part uses a ping-pong type argument that originates in the reference \cite{MR3228932}. Other precursors to this idea can be found in work on strong approximation in \cite{MR1090825}, on maximal subgroups of $\mathrm{PSL}_d(\mathbb{Z})$ in \cite{MR3569567} and on the congruence subgroup property in \cite{MR0171856}. We need that $d\ge 5$ only because this condition guarantees that all $\mathrm{PSL}_2(\mathbb{Z})$ orbits in $\mathrm{PSL}_d(\mathbb{Z}/p\mathbb{Z})$ have density bounded by a constant which is strictly less than 1. We do not know whether the result can be improved to $d \in \{3,4\}$. \\
\\
Because Theorem \ref{thm} uses such heavy machinery, it is natural to wonder whether results of its type can be found among other groups. For example, if $H$ is a direct product of two free groups then do there exist subgroups $A$ and $B$ satisfying the criteria of Theorem \ref{thm:general}? What if $H$ is a lattice in the isometry group of quaternionic hyperbolic space? Another interesting case would be to establish Theorem \ref{thm:general} for a $2$-Kazhdan group such as a higher-rank $p$-adic lattice. The relevance of this last case is that in \cite{2017arXiv171110238D} it is shown that $2$-Kazhdan groups satisfy the analog of flexible stability for homomorphisms into finite-dimensional unitary groups with the unnormalized Frobenius metric. It is unknown whether $\mathrm{PSL}_d(\mathbb{Z})$ is $2$-Kazhdan.

\subsection{Acknowledgments}

The first author would like to thank Tsachik Gelander for conversations on this problem many years ago, including a proof that certain groups constructed in a manner similar to the group appearing in the main theorem are not residually finite and therefore are good candidates for being nonsofic. We thank Emmanuel Breuillard for helpful discussions related to the proof of Lemma \ref{lem.small}. We also thank Yves Stalder for catching some errors in a previous version.

\section{General results}\label{sec:general}

\begin{thm}\label{thm:general}
Suppose $H$ is a flexibly stable countable discrete group with subgroups $A$ and $B$ satisfying the following conditions.
\begin{enumerate}[label={(\arabic*)}]
\item If $K \le H$ has finite index, then every $B$-orbit in $H/K$ is contained in an $A$-orbit. Explicitly, this means for every $h \in H$ we have $BhK \subseteq AhK$.
\item If $C$ is the subgroup generated by $A$ and $B$ then there is an automorphism $\omega \in \mathrm{Aut}(C)$ such that $\omega(A)=B$ and $\omega^2$ is the identity.
\item There is a constant $\l>1$ such that if $K$ is a proper finite index subgroup of $H$ then for every $g,h \in H$ we have
\begin{equation} \label{eq.becky-6} |AgK| \ge \l |BhK| \end{equation}
where the cardinality $|\cdot|$ is taken in $H/K$. 
\item $A$ has property $(\tau)$ with respect to the family of finite index subgroups \[ \{K \cap A:~ K \le H, [H:K]<\infty\}. \] \end{enumerate}
Then the group
\[ G =\langle H, t | t^2=1, tat^{-1}=\omega(a)~ \forall a\in A\rangle \] is not sofic.
\end{thm}

The proof of Theorem \ref{thm:general} is in Subsection \ref{sec:general-proof} below after some preliminaries.

\subsection{Property ($\tau$)}

This section reviews Property $(\tau)$.

\begin{definition} Let $\Gamma = (V,E)$ be a finite graph. If $W \subseteq V$ the \textbf{edge boundary} in $\Gamma$ of $W$ will be denoted $\partial_\Gamma W$ and consists of all edges $(v,w) \in E$ where $v \in W$ and $w \notin W$. If $W$ is nonempty the \textbf{edge isoperimetric ratio} of $W$ will be denoted $\iota_\Gamma(W)$ and is defined to be $|\partial_\Gamma W|\, |W|^{-1}$. The \textbf{edge expansion constant} of $\Gamma$ will be denoted $e(\Gamma)$ and is defined to be the minimum value of $\iota_\Gamma(E)$ over all nonempty subsets $W \subseteq V$ satisfying $|W| \leq \frac{1}{2}|V|$. \end{definition}

\begin{definition} Let $(\Gamma_n)_{n=1}^\infty$ be a sequence of finite connected graphs and let $c>0$. We say that $(\Gamma_n)_{n=1}^\infty$ forms a family of $c$\textbf{-expanders} if $\inf_{n\in \mathbb{N}} e(\Gamma_n) \geq c$. We say that $(\Gamma_n)_{n=1}^\infty$ forms a family of \textbf{expanders} if it forms a family of $c$-expanders for some $c>0$. \end{definition}

\begin{definition} Let $G$ be a group, $H \le G$ and $S \subseteq G$. The {\bf Schreier coset graph} $\Sch(G/H,S)$ is the multi-graph with vertex set $G/H$ and edges $\{gH, sgH\}$ for all $gH \in G/H$ and $s\in S$. Multiple edges and self-loops are allowed.
\end{definition}

\begin{definition} A group $G$ has {\bf Property ($\tau$) with respect to a family $\mathcal{F}$} of finite index subgroups of $G$ if there is a finite generating set $S \subseteq G$ and a constant $c>0$ such that for every $H \in \mathcal{F}$ we have that $\Sch(G/H,S)$ is a $c$-expander. \end{definition}

It is easy to see that Property ($\tau$) for a family $\mathcal{F}$ is does not depend on the choice of $S$.

\subsection{Modular HNN extensions} \label{subsec.HNN}

Let $H$ be a countable discrete group with subgroups $A,B \le H$ and suppose there is an isomorphism $\phi:A \to B$. The {\bf HNN extension} $H \ast_\phi$ is defined to be $(H \ast \langle t \rangle)/ N$ where $H \ast \langle t\rangle$ is the free product of $H$ with a copy of $\mathbb{Z}$ and $N$ is the smallest normal subgroup of $H \ast \langle t \rangle$ containing all elements of the form $t a t^{-1} \phi(a)^{-1}$ for $a \in A$. We will need a $\mathrm{mod} \,\,2$ version of the construction above. So let $N_2$ be the smallest normal subgroup of $H \ast \langle t \rangle$ containing all elements of the form $tat^{-1}\phi(a)^{-1}$ for $a \in A$ along with $t^2$. Let $H \ast_\phi/2 = (H \ast \langle t \rangle)/N_2$. 

\begin{lemma}\label{lem:HNN}
Let $C$ be the subgroup of $H$ generated by $A$ and $B$. Assume there exists an automorphism $\omega$ of $C$ such that $\omega^2$ is the identity and such that $\omega(a) = \phi(a)$ for all $a \in A$ and $\omega(b) = \phi^{-1}(b)$ for all $b \in B$. Then the canonical homomorphism from $H$ to $H \ast_\phi /2$ is injective. \label{lem.inj}
\end{lemma}

\begin{proof}[Proof of Lemma \ref{lem.inj}] Let $D$ be the semidirect product $C \rtimes \mathbb{Z}/2\mathbb{Z}$ where $\mathbb{Z}/2\mathbb{Z}$ acts on $C$ via the automorphism $\omega$. We claim that $H \ast_\phi/2$ can be constructed as the free product of $H$ with $D$ amalgamated over the common subgroup $C$. Indeed, $H \ast_C D$ is naturally generated by $H$ and the additional generator $t = t^{-1}$ of $\mathbb{Z}/2\mathbb{Z}$. If $a \in A$ then $tat$ is equal to $\omega(a) = \phi(a)$ and similarly if $b \in B$ then $tbt$ is equal to $\omega(b) = \phi^{-1}(b)$. Therefore $tat\phi(a)^{-1}$ and $tbt\phi^{-1}(b)^{-1}$ are trivial in $H \ast_C D$ for all $a \in A$ and all $b \in B$. By the universal property of free products with amalgamation we see that these relations suffice to describe $H \ast_C D$ and so we have established the claim. Since the factor groups always inject into an amalgamated free product this completes the proof of Lemma \ref{lem.inj}. \end{proof}

\subsection{Proof of Theorem \ref{thm:general}}\label{sec:general-proof}

We now prove Theorem \ref{thm:general}. By Lemma \ref{lem:HNN}, the canonical homomorphism from $H$ into $G$ is injective. Thus we identify $H$ as a subgroup of $G$ from now on. Assume toward a contradiction that there exists a sofic approximation $\Sigma = (\sigma_n:G \to \mathrm{Sym}(V_n))_{n=1}^\infty$ to $G$.  Since $H$ is flexibly stable, we may assume without loss of generality that the restriction of $\Sigma$ to $H$ is perfect. \\
\\
Since $A$ has property ($\tau$) with respect to the family \[ \{K \cap A:~ K \le H, [H:K]<\infty\} \] there exists a finite generating set $S \subseteq A$ and a constant $c>0$ such that for every finite index subgroup $K$ of $H$ all connected components of the Schreier coset graph $\Sch(H/K,S)$ are $c$-expanders. Let $\Gamma_n$ be the graph on $V_n$ corresponding to $\{\sigma_n^s:~s\in S\}$. Explicitly, this means that the edges of $\G_n$ are the pairs $\{v,\s_n^s(v)\}$ for $v \in V_n$ and $s\in S$. Since $\Sigma \rest H$ is perfect, every connected component of $\Gamma_n$ is a Schrieier coset graph of the form $\Sch(H/K,S)$ and so every connected component of $\Gamma_n$ is $c$-expander.\\
\\
Let $\Lambda_n$ be the graph on $V_n$ corresponding to $\{\sigma_n^{\omega(s)}:~s\in S\}$. The hypothesis that every $B$-orbit is contained in an $A$-orbit implies that the vertex set of every $\Lambda_n$-connected component is contained in the vertex set of a $\Gamma_n$-connected component.  \\
\\
For the remainder of Subsection \ref{sec:general-proof} we fix $n \in \mathbb{N}$ such that $\sigma_n$ is a sufficiently good sofic approximation for certain conditions stated later to hold. We suppress the subscript $n$ in notations.\\
\\
Let $\Omega_1,\ldots,\Omega_m$ be an enumeration of the connected components of $\Gamma$ such that $|\Omega_j| \geq |\Omega_{j+1}|$ for all $j \in \{1,\ldots,m-1\}$. Let $D$ be the set of all $w \in V$ such that $w \in \Omega_j$ and $\sigma^t(w) \in \Omega_k$ where $j \leq k$. If $\sigma$ is a sufficiently good sofic approximation then for at least $\frac{9}{10}|V|$ vertices $w \in V$ we must have that $(\sigma^t)^2(w) = w$. If the last condition is satisfied then at least one of $w$ and $\sigma^t(w)$ is an element of $D$. Therefore $|D| \geq \frac{9}{20}|V|$. \\
\\
Let $\mathcal{I} \subseteq \{1,\ldots,m\}$ be the set of all indexes $j$ such that $|D \cap \Omega_j| \geq \frac{1}{10}|\Omega_j|$. We must have \begin{equation} \label{eq.thalia-2} \sum_{j \in \mathcal{I}} |\Omega_j| \geq \frac{|V|}{10}. \end{equation} Fix $j \in \mathcal{I}$ and consider the set $D \cap \Omega_j$. Let $\Theta_1,\ldots,\Theta_r$ be the partition of $D \cap \Omega_j$ into the intersections of $D \cap \Omega_j$ with $\sigma^t$-pre-images of connected components of $\Lambda$. Let $q \in \{1,\ldots,r\}$ and suppose $\sigma^t(\Theta_q) \subseteq \Omega_k$. Since $\sigma^t(\Theta_q)$ is contained in a single connected component of $\Lambda$ we see from (\ref{eq.becky-6}) that $\l |\sigma^t(\Theta_q)| \leq |\Omega_k|$. Since $\Theta_q \subseteq D$ we have $j \leq k$ so that $|\Omega_j| \geq |\Omega_k|$ and therefore $\l |\Theta_q| \leq  |\Omega_j|$ which implies $(\l-1)|\Theta_q| \le |\Omega_j \setminus \Theta_q|$.\\
\\
Since $\Theta_q \subseteq\Omega_j$ we have $\partial_\Gamma \Theta_q = \partial_\Gamma (\Omega_j \setminus \Theta_q)$. Since every connected component of $\Gamma$ is a $c$-expander,
$$|\partial_\Gamma \Theta_q| \ge c \min\left( |\Theta_q|, |\Omega_j \setminus \Theta_q|\right) \ge c \min(1, \l-1) |\Theta_q|.$$
\\
Let $c' = c\min(1, \l-1)$. So \begin{equation} |\partial_\Gamma \Theta_1 \cup \cdots \cup \partial_\Gamma \Theta_r| \geq \frac{c'}{2}(|\Theta_1| + \cdots + |\Theta_r|) = \frac{c'}{2}|D \cap \Omega_j| \geq \frac{c'}{20}|\Omega_j| \label{eq.thalia-1} \end{equation} Here the first inequality holds because the pairwise disjointness of the $\Theta_q$ guarantees that for any edge $e$ there are at most two indices $q$ such that $e \in \partial_\Gamma \Theta_q$.\\
\\
Let $q \in \{1,\ldots,r\}$, let $v \in \Theta_q$ and suppose $(v,w)$ is an edge in $\partial_\Gamma \Theta_q$. If $w \notin D$ then $\sigma^t(v)$ and $\sigma^t(w)$ are in different connected components of $\Gamma$, and so in particular they are in different connected components of $\Lambda$. On the other hand, if $w \in D$ then by hypothesis $\sigma^t(v)$ and $\sigma^t(w)$ are in different connected components of $\Lambda$. Hence in either case $(\sigma^t(v),\sigma^t(w))$ is not an edge in $\Lambda$. \\
\\
From (\ref{eq.thalia-1}) we see that for at least $\frac{c'}{20}|\Omega_j|$ edges $(v,w)$ in $\Gamma \rest \Omega_j$ the image $(\sigma^t(v),\sigma^t(w))$ is not an edge in $\Lambda$. Summing this over all $j \in \mathcal{I}$ we see from (\ref{eq.thalia-2}) that there is a set $K$ of edges in $\Gamma$ with $|K| \geq \frac{c'}{200}|V|$ such that for each $(v,w) \in K$ the image $(\sigma^t(v),\sigma^t(w))$ is not an edge in $\Lambda$. However, if $\sigma$ is a sufficiently good sofic approximation then the number of such edges should be an arbitrarily small fraction of $|V|$. Thus we have obtained a contradiction and the proof of Theorem \ref{thm:general} is complete.

\section{Subgroups of special linear groups} \label{sec.prelim}

In Section \ref{sec.prelim} we will prove that $\mathrm{PSL}_d(\mathbb{Z})$ satisfies the conditions of Theorem \ref{thm:general} for $d\ge 5$, thereby completing the proof of Theorem \ref{thm}

\subsection{Ping-pong arguments}

The next Lemma constructs the subgroups $A$ and $B$ that will be used in our application of Theorem \ref{thm:general}. 

\begin{lemma}\label{prop:ping-pong}
Let $d\ge 3$. Identify $\mathrm{PSL}_2(\mathbb{Z})$ with the image in $\mathrm{PSL}_d(\mathbb{Z})$ of the copy of $\mathrm{SL}_2(\mathbb{Z})$ in the upper left corner of $\mathrm{SL}_d(\mathbb{Z})$. Then there exist subgroups $A$ and $B$ of $\mathrm{PSL}_d(\mathbb{Z})$ such that the following hold.
\begin{enumerate}[label={(\arabic*)}]
\item $A$ and $B$ are free groups of rank 4,
\item $A$ is profinitely dense in $\mathrm{PSL}_d(\mathbb{Z})$,
\item $B$ is contained in $\mathrm{PSL}_2(\mathbb{Z})$ and
\item the subgroup $\langle A,B\rangle$ is free of rank 8.
\end{enumerate}
\end{lemma}

 \begin{proof}[Proof of Lemma \ref{prop:ping-pong}]
 By the main theorem of \cite{MR3228932} there exists a profinitely dense free subgroup $A$ of $\mathrm{PSL}_d(\mathbb{Z})$ with rank 4. The construction of this subgroup gives additional information about $A$ that we will use. To describe this, we recall the following notions from \cite{MR3228932}. \\
 \\
An element $g \in \mathrm{PSL}_d(\mathbb{Z})$ is {\bf hyperbolic} if it is semisimple, admits a unique (counting multiplicities) eigenvalue of maximum absolute value and a unique eigenvalue of minimum absolute value. Let $\{v_1,v_2,\ldots, v_d\}$ be a basis of generalized eigenvectors such that $v_1$ corresponds to the unique maximal eigenvalue of $g$ and $v_n$ corresponds to the unique minimal eigenvalue. Let $\alpha(g)=[v_1] \in \R\P^{d-1}$ and $\rho(g) = [\textrm{span}(v_2,\ldots,v_d)] \subseteq\R\P^{d-1}$. These are the attracting fixed point and repelling hyperplane of $g$. Note that $\alpha(g^{-1}) = [v_d]$ and $\rho(g^{-1}) =  [\textrm{span}(v_1,\ldots,v_{d-1})].$ Although $g$ need not be diagonalizable, $\rho(g)$ does not depend on the choice of basis $\{v_1,\ldots, v_d\}$. 
 
\begin{definition} \label{def.root} Let $g_0,g_1,\ldots, g_s \in \mathrm{PSL}_d(\mathbb{Z})$ be hyperbolic elements. Then $\{g_1,\ldots, g_s\}$ is a {\bf $g_0$-rooted free system} if there exist open sets $O_j \subseteq\R\P^{d-1}$ for $j\in \{0,1,\ldots, s\}$ such that the following hold.
 \begin{enumerate}[label={(\arabic*)}]
 \item The sets $\{O_j\}_{j=0}^s$ are pairwise disjoint, 
 \item for all $j \in \{1,\ldots,s\}$ we have  \begin{equation} \label{eq.coke-8} \{\a(g_j) , \a(g_j^{-1})\} \subseteq O_j \subseteq \overline{O_j} \subseteq \R \P^{d-1} \setminus (\rho(g_0) \cup \rho(g_0^{-1})) \end{equation}
 \item we have \[ \{\a(g_0), \a(g_0^{-1})\} \subseteq O_0 \subseteq \overline{O_0} \subseteq\R\P^{d-1} \setminus \left(\bigcup_{j=1}^s \rho(g_j)\cup\rho(g_j^{-1})\right) \]
 \item and $g_j(\overline{O_k}) \cup g_j^{-1}(\overline{O_k}) \subseteq O_j$ for all distinct pairs $j,k \in \{0,\ldots, s\}$.
 \end{enumerate} \end{definition}
 
The construction in \cite{MR3228932} shows that there exist hyperbolic elements $g_0,g_1,g_2,g_3,g_4 \in \mathrm{PSL}_d(\mathbb{Z})$ such that $\{g_1,g_2,g_3,g_4\}$ is a $g_0$-rooted free system and the subgroup $\langle g_1,g_2,g_3,g_4\rangle$ is profinitely dense. (The fourth clause in the definition of a $g_0$-rooted free system that we use differs slightly from the one used in \cite{MR3228932}. However, it is easy to verify that their proof gives a $g_0$-rooted free system in our sense.) We make the following claim.
 
\begin{claim} \label{cla-1} After conjugating the elements above if necessary, we may assume that $\rho(g_j)\cup \rho(g_j^{-1})$ does not contain $[\textrm{span}(e_1,e_2)]$ and $\alpha(g_j)\cup\alpha(g_j^{-1})$ does not intersect $[\textrm{span}(e_3,\ldots,e_d)]$ for any $j \in \{0,\ldots,4\}$. \end{claim}
 
\begin{proof}[Proof of Claim \ref{cla-1}] We claim that if $V,W \subset \mathbb{R}^d$ are subspaces, then the set of all elements $h \in \mathrm{SL}_d(\mathbb{R})$ such that $h(V) \subset W$ is Zariski-closed. Since intersections of closed sets are closed, it suffices to consider the special case in which $V$ is spanned by a single vector $v$. Let $\pi:\mathbb{R}^d \to W$ be orthogonal projection. Then $h \mapsto \|hv- \pi(hv) \|^2_2$ is a polynomial whose zero set is exactly the set of $h \in \mathrm{SL}_d(\mathbb{R})$ such that $hv \in W$. This proves the claim.\\
\\
By abuse of notation, we identify $\rho(g_j)$ and $\alpha(g_j)$ with their associated subspaces in $\mathbb{R}^d$. Let $V_j$ be the set of all $h \in   \mathrm{SL}_d(\mathbb{R})$ such that $h(\rho(g_j)\cup \rho(g_j^{-1}))$ does not contain $\textrm{span}(e_1,e_2)$.  Let $W_j$ be the set of all $h \in  \mathrm{SL}_d(\mathbb{R})$ such that $h(\alpha(g_j)) \notin \textrm{span}(e_3,\ldots,e_d)$ and $h(\alpha(g^{-1}_j)) \notin \textrm{span}(e_3,\ldots,e_d)$. By the previous paragraph both $V_j$ and $W_j$ are Zariski open. They are non-empty because for every $k$, the group $\mathrm{SL}_d(\mathbb{R})$ acts transitively on the set of $k$-dimensional subspaces. \\
\\
Since $\mathrm{SL}_d(\mathbb{R})$ is connected, it is Zariski-connected. So the intersection of any finite collection of non-empty Zariski open subsets is non-empty. In particular, $\bigcap_{j=0}^4(V_j \cap W_j)$ is nonempty and Zariski-open. By the Borel density Theorem, $\mathrm{SL}_d(\mathbb{Z}) $ is Zariski dense in $\mathrm{SL}_d(\mathbb{R})$. Therefore the set \begin{equation} \label{eq.becky-5} \mathrm{SL}_d(\mathbb{Z}) \cap \left( \bigcap_{j=0}^4 (V_j \cap W_j) \right) \end{equation} is non-empty. Let $h \in \mathrm{PSL}_d(\mathbb{Z})$ be the image of an element of the set in (\ref{eq.becky-5}). Replacing each $g_j$ with $hg_jh^{-1}$ proves Claim \ref{cla-1}.

\end{proof}

Let $\pi: \R\P^{d-1} \setminus [\textrm{span}(e_3,\ldots,e_d)]\to [\textrm{span}(e_1,e_2)]$ denote the projection map. Claim   \ref{cla-1} shows that the points $\pi(\a(g^{\pm 1}_j))$ are well-defined.\\
\\
It is well-known that given any finite subset $F$ of $\R\P^1$, there exists a hyperbolic element $f \in \mathrm{PSL}_2(\mathbb{Z}) $ which has no fixed point in $F$. Therefore we can find hyperbolic elements $h_1,h_2,h_3,h_4  \in \mathrm{PSL}_2(\mathbb{Z}) $ such that the following hold.
\begin{align} 
 \left( \bigcup_{j=1}^4 \{\alpha(h_j),\alpha(h_j^{-1})\} \right) &\cap \left( \bigcup_{k=0}^4 \pi\left(\left\{\a(g_k), \a\left(g_k^{-1}\right)\right\} \right) \right) = \emptyset \label{eq.coke-4.23} \\
 \{\a(h_j),\a(h_j^{-1})\} & \cap \{\a(h_k), \a(h_k^{-1})\} = \emptyset \label{eq.coke-5}\\
 \{\alpha(g_j),\alpha(g_j^{-1})\} &\cap \{\alpha(h_k),\alpha(h_k^{-1})\} = \emptyset \label{eq.coke-6}\\
 \left( \bigcup_{j=1}^4 \{\a(h_j), \a(h_j^{-1})\} \right)  &\cap \left(\bigcup_{k=0}^4 (\rho(g_k)\cup \rho(g_k^{-1}))\right) = \emptyset.\label{eq.coke-3}
\end{align}

Here, the set in (\ref{eq.coke-5}) is empty for all distinct pairs $j,k \in \{1,2,3,4\}$ and the set in (\ref{eq.coke-6}) is empty for all $j \in \{0,\ldots,4\}$ and all $k \in \{1,\ldots,4\}$. Equation (\ref{eq.coke-3}) is justified by the first part of Claim  \ref{cla-1} which implies $\left(\bigcup_{k=0}^4 (\rho(g_k)\cup \rho(g_k^{-1}))\right) \cap [\textrm{span}(e_1,e_2)]$ is finite. \\
\\
Let $j \in \{1,\ldots, 4\}$ and $k\in \{0,\ldots, 4\}$. Since $h_j \in \mathrm{PSL}_2(\mathbb{Z})$, $\rho(h_j)$ is the span of $\a(h_j^{-1}) \in [\textrm{span}(e_1,e_2)]$ and  $[\textrm{span}(e_3,\ldots, e_d)]$. So the projection $\pi$ maps $\rho(h_j) \setminus [\textrm{span}(e_3,\ldots,e_d)]$ to the point $\a(h_j^{-1})$. So (\ref{eq.coke-4.23}) implies

\begin{equation} \left( \bigcup_{j=1}^4 (\rho(h_j) \cup \rho(h_j^{-1})) \right)  \cap   \left(\bigcup_{k=0}^4 \{\a(g_k), \alpha(g_k^{-1})\}\right) = \emptyset. \label{eq.coke-4} \end{equation}

Let $\Delta$ be the standard metric on $\mathbb{RP}^{d-1}$ and for $\epsilon > 0$ and $S \subseteq \mathbb{RP}^{d-1}$ write \[ \mathrm{B}_\epsilon(S) = \{v \in \mathbb{RP}^{d-1}:\Delta(v,S) < \epsilon \}  \] Observe that for any hyperbolic element $\ell \in \mathrm{PSL}_2(\mathbb{Z})$, there exists a sequence $(\epsilon_n)_{n=1}^\infty$ of positive numbers decreasing to zero such that \begin{equation} \label{eq.coke-1} \ell^n\bigl(\mathbb{RP}^{d-1} \setminus \mathrm{B}_{\epsilon_n}(\rho(\ell)\cup \rho(\ell^{-1})) \bigr) \subseteq \mathrm{B}_{\epsilon_n}(\{\alpha(\ell),\alpha(\ell^{-1}\}). \end{equation}
So after replacing $g_0$ with $g_0^n$ for a sufficiently large $n$ we may replace $O_0$ with a smaller open neighborhood of $\{\alpha(g_0),\alpha(g_0^{-1})\}$. Thus using (\ref{eq.coke-4}) we may assume that 
\begin{equation} 
\overline{O_0} \cap \left( \bigcup_{j=1}^4 (\rho(h_j)\cup\rho(h_j^{-1}))\right) = \emptyset.  \label{eq.coke-18} 
\end{equation}

By (\ref{eq.coke-5}, \ref{eq.coke-3}) there exist open neighborhoods $U_1,U_2,U_3,U_4$ of $\{\alpha(h_1),\alpha(h_1)^{-1}\},\ldots,\{\alpha(h_4),\alpha(h_4^{-1})\}$ respectively such that for each distinct pair $j,k$ in $\{1,2,3,4\}$,
\begin{equation} \overline{U_j} \cap (\rho(h_k) \cup \rho(h_k^{-1}) \cup \rho(g_0) \cup \rho(g_0^{-1})) = \emptyset.
\label{eq.coke-19} \end{equation}

By (\ref{eq.coke-1}, \ref{eq.coke-18}, \ref{eq.coke-19}) there exists $N \in \mathbb{N}$ such that 
\begin{equation} h_j^N(\overline{U_k} \cup \overline{O_0}) \cup h_j^{-N}(\overline{U_k} \cup \overline{O_0}) \subseteq U_j \label{eq.coke-20} \end{equation} for all distinct pairs $j,k \in \{1,\ldots,4\}$. By (\ref{eq.coke-1}, \ref{eq.coke-19}) there exists $M \in \mathbb{N}$ such that 
\begin{equation} g_0^M(\overline{U_k} \cup \overline{O_k}) \cup g_0^{-M}(\overline{U_k} \cup \overline{O_k}) \subseteq O_0 \label{eq.coke-2end} \end{equation} for all distinct pairs $k \in \{1,\ldots,4\}$. Write $s_j = h_j^N$ for $j \in \{1,\ldots,4\}$ and $s_j = g_0^{-M}g_{j-4}g_0^M$ for $j \in \{5,\ldots,8\}$. Also write $U_j = g_0^{-M}(O_{j-4})$ for $j \in \{5,\ldots,8\}$.\\
\\
We claim that the set $\{s_1,\ldots,s_8\}$ generates a free group of rank $8$. This will suffice to establish Lemma \ref{prop:ping-pong} since the elements $\{s_5,\ldots,s_8\}$ generate a profinitely dense subgroup of $\mathrm{PSL}_d(\mathbb{Z})$ and since $\{s_1,\ldots,s_4\} \subseteq \mathrm{PSL}_2(\mathbb{Z})$. We claim that \begin{equation} \label{eq.coke-11} (s_j(\overline{U_k}) \cup s_j^{-1}(\overline{U_k})) \subseteq U_j \end{equation} for all distinct pairs $j,k \in \{1,\ldots,8\}$. By the standard ping-pong lemma from \cite{MR0286898} this suffices to establish the claim. \\
\\
Suppose first that $j,k \in \{1,\ldots,4\}$. Then (\ref{eq.coke-11}) follows from (\ref{eq.coke-20}). Suppose now that $j \in \{1,\ldots,4\}$ and $k \in \{5,\ldots,8\}$. The definitions of $s_j$ and $U_k$ imply (\ref{eq.coke-11}) is equivalent to $h_j^{\pm N} g_0^{-M} \overline{O_{k-4}} \subset U_j$. This is true because (\ref{eq.coke-2end}) implies $g_0^{-M} \overline{O_{k-4}} \subset O_0$ and (\ref{eq.coke-20}) implies $h_j^{\pm N} O_0 \subset U_j$.\\
\\
Suppose now that $k \in \{1,\ldots,4\}$ and $j \in \{5,\ldots,8\}$. The definitions of $s_j$ and $U_k$ imply (\ref{eq.coke-11}) is equivalent to $g_0^{-M}g_{j-4}^{\pm 1} g_0^{M} \overline{U_{k}} \subset g_0^{-M}O_{j-4}$. This is true because (\ref{eq.coke-2end}) implies $g_0^{M} \overline{U_k} \subset O_0$ and the fourth item in Definition \ref{def.root} implies $g_{j-4}^{\pm 1} O_0 \subset O_{j-4}$. \\
\\
Finally, suppose $j,k \in \{5,\ldots,8\}$. The definitions of $s_j$ and $U_k$ imply (\ref{eq.coke-11}) is equivalent to $$g_0^{-M}g_{j-4}^{\pm 1} g_0^{M} g_0^{-M}\overline{O_{k-4}} \subset g_0^{-M}O_{j-4}.$$
This simplifies to $g_{j-4}^{\pm 1}\overline{O_{k-4}} \subset O_{j-4}$. This is true by the fourth item in Definition \ref{def.root}. This completes the proof of Lemma \ref{prop:ping-pong}.  \end{proof}

\subsection{Expansion in quotients of $\mathrm{PSL}_d(\mathbb{Z})$} \label{subsec.expand}

\begin{lemma} \label{lem.right} Let $d \geq 3$. Let $A$ be a finitely generated profinitely dense subgroup of $\mathrm{PSL}_d(\mathbb{Z})$. Then $A$ has property $(\tau)$ with respect to the family \[ \{K \cap A:~ K \le \mathrm{PSL}_d(\mathbb{Z}), [\mathrm{PSL}_d(\mathbb{Z}):K]<\infty\} \] \end{lemma}

\begin{proof} Because $A$ is profinitely dense, it is Zariski dense. Let $S \subseteq A$ be a finite generating set. Write $\pi_q:\mathrm{PSL}_d(\mathbb{Z}) \twoheadrightarrow \mathrm{PSL}_d(\mathbb{Z}/q\mathbb{Z})$ for reduction modulo $q$. Theorem 1 in \cite{MR2897695} asserts that the Cayley graphs of $\pi_q(A)$ with respect to $S$ form a family of $c$-expanders for some $c>0$. Let $K\le  \mathrm{PSL}_d(\mathbb{Z})$ have finite index. By the congruence subgroup property as established in \cite{MR0171856}, there exists a $q \in \mathbb{N}$ such that $K$ contains the kernel $\G_q$ of the natural surjection $\mathrm{PSL}_d(\mathbb{Z}) \twoheadrightarrow \mathrm{PSL}_d(\mathbb{Z}/q\mathbb{Z}).$ It follows that the quotient map $\pi_q(A) \twoheadrightarrow A/(A \cap K)$ induces a covering space \[ \Sch(\pi_q(A), S) \twoheadrightarrow \Sch(A/(A \cap K), S) \]  Therefore the preimage of a subset $D$ of $A/(A \cap K)$ has the same edge isoperimetric ratio as $D$. Since $\Sch(\pi_q(A), S)$ is a $c$-expander, so is $\Sch(A/(A \cap K), S)$.  \end{proof}

\subsection{Bounds on the density of $\mathrm{PSL}_2(\mathbb{Z})$-orbits in finite quotients of $\mathrm{PSL}_d(\mathbb{Z})$}

The main result of this section is Lemma \ref{lem.small}, which provides an upper bound on densities of $\mathrm{PSL}_2(\mathbb{Z})$-orbits in finite quotients of $\mathrm{PSL}_d(\mathbb{Z})$. First we prove two lemmas that allow us to reduce the general case to the  $\mathrm{PSL}_d(\mathbb{Z}/p\mathbb{Z})$ case.

\begin{lemma}\label{semisimple}
Let $G=H_1 \times \cdots \times H_k$ be a direct product of simple groups $H_i$. Identify each group $H_i$ with the corresponding subgroup of $G$. For each subset $S \subset \{1,\ldots, k\}$ let $H_S\le G$ be the product of the   subgroups $H_i$ over $i\in S$ and let $\pi_S:G \to H_S$ be the projection. Also write $\pi_i = \pi_{\{i\}}$. Suppose that for each $i\ne j$, $H_i$ is not isomorphic to $H_j$. Then the following statements hold.
\begin{enumerate}
\item If $N \le G$ is normal then there exists a subset $S$ such that $N=H_S$.
\item If $K \le G$ is maximal then there exists an index $i$ such that $\pi_i(K)< H_i$ is a proper subgroup of $H_i$. 
\end{enumerate}

\end{lemma}

\begin{proof}

We prove this lemma by induction on $k$. The base case ($k=1$) is vacuous. For $S \subset \{1,\ldots, k\}$, let $S^c = \{1,\ldots, k\} \setminus S$ be the complement.\\ 
\\
Suppose, for some proper $S \subset \{1,\ldots, k\}$ that $N\cap H_S$ is nontrivial. The induction hypothesis applied to the inclusion $N \cap H_S \le H_S$ implies that $N \cap H_S = H_T$ for some nonempty $T \subset S$. Apply the induction hypothesis to the inclusion $N/H_T \le G/H_T \cong H_{T^c}$ to obtain a subset $R$ with $T \subset R \subset \{1,\ldots, k\}$ such that $N/H_T = H_R/H_T$. Therefore $N=H_R$. This proves the lemma in the special case that $N \cap H_S$ is nontrivial for some proper $S$. So we may assume without loss of generality that $N \cap H_S$ is trivial for all proper $S$. \\
\\
Without loss of generality, we assume $N$ is nontrivial. Since the kernel of $\pi_i$ is $H_{\{i\}^c}$ and $H_{\{i\}^c} \cap N$ is trivial, the restriction of $\pi_{i}$ to $N$ is injective. Since $H_i$ is simple and $\pi_{i}(N) \le H_i$ is a nontrivial normal subgroup, this implies that $N$ surjects onto $H_i$ for all $i$. Thus $N$ is isomorphic to $H_i$ for all $i$. In particular, there exist $i\ne j$ such that $H_i$ is isomorphic to $H_j$. This contradiction proves the first part of the lemma. \\
\\
Now suppose $K \le G$ is maximal. Let $S \subsetneq \{1,\ldots, k\}$ be proper. If $\pi_{S}(K) \ne H_{S}$ then apply the induction hypothesis to the inclusion $\pi_{S}(K) \le H_{S}$ to obtain the lemma. So we may assume without loss of generality that $\pi_{S}(K) = H_{S}$ for all proper subsets $S$. \\
\\
Again, let $S \subsetneq \{1,\ldots, k\}$ be proper. Since $H_{S}$ is normal in $G$, $H_{S} \cap K$ is normal in $K$. Since $\pi_{S}$ maps $K$ onto $H_{S}$, this implies $\pi_S(H_{S} \cap K)$ is normal in $H_{S}$. However, $\pi_{S}$ restricts to the identity on $H_{S}$. So $H_{S} \cap K$ is normal in $H_{S}$. So the first part of this lemma implies $H_{S} \cap K = H_{T}$ for some subset $T \subset S$. \\
\\
Suppose $K$ contains $H_T$ for some non-empty $T \subset \{1,\ldots, k\}$. Then we can apply the induction hypothesis to the inclusion $K/H_T \le G/H_T \cong H_{T^c}$ to finish the lemma. Therefore, we may assume without loss of generality that $K$ does not contain $H_T$ for any non-empty $T$. The previous paragraph now implies $H_{S} \cap K$ is trivial for all proper subsets $S$. Therefore, the map $\pi_{i}:G \to H_i$ is injective on $K$. Since it is also surjective (by the second paragraph before this one), this implies $K$ is isomorphic to $H_i$ for all $i$. In particular there are different indices $i\ne j$ such that $H_i$ is isomorphic to $H_j$. This completes the proof. 

\end{proof}

\begin{lemma} \label{lem.proper} Let $q \in \mathbb{N}$ and let $K \leq \mathrm{PSL}_d(\mathbb{Z}/q\mathbb{Z})$ be a proper subgroup. Then there is a prime factor $p$ of $q$ such that the image of $K$ under reduction $\mathrm{mod} \,\,p$ is a proper subgroup of $\mathrm{PSL}_d(\mathbb{Z}/p\mathbb{Z})$. \end{lemma}

\begin{proof}
 It suffices to consider the special case in which $K$ is a maximal proper subgroup. Suppose toward a contradiction that the proposition fails for $K$. We may assume without loss of generality that $q$ has the minimal number of distinct prime factors among all $r \in \mathbb{N}$ such that the proposition fails for some subgroup of $\mathrm{PSL}_d(\mathbb{Z}/r\mathbb{Z})$. \\
\\
Recall that if $G$ is a finite group then the Frattini subgroup $\Phi(G)$ is the intersection of all maximal proper subgroups of $G$. If $G$ and $H$ are finite groups we have $\Phi(G \times H) = \Phi(G) \times \Phi(H)$.\\
\\
Let $q = p_1^{n_1}\cdots p_k^{n_k}$ be the prime factorization of $q$. By the Chinese remainder theorem, we have that $\mathrm{PSL}_d(\mathbb{Z}/q\mathbb{Z})$ is isomorphic to $\mathrm{PSL}_d(\mathbb{Z}/p_1^{n_1}\mathbb{Z}) \times \cdots \times \mathrm{PSL}_d(\mathbb{Z}/p_k^{n_k}\mathbb{Z})$. So $\Phi(\mathrm{PSL}_d(\mathbb{Z}/q\mathbb{Z}))$ is isomorphic to $\Phi(\mathrm{PSL}_d(\mathbb{Z}/p_1^{n_1}\mathbb{Z})) \times \cdots \times \Phi(\mathrm{PSL}_d(\mathbb{Z}/p_k^{n_k}\mathbb{Z}))$.\\
\\
We claim that for any prime $p$ and natural number $n$, the Frattini subgroup $\Phi(\mathrm{PSL}_d(\mathbb{Z}/p^n\mathbb{Z}))$ is the kernel of the surjective homomorphism $\mathrm{PSL}_d(\mathbb{Z}/p^n\mathbb{Z}) \to \mathrm{PSL}_d(\mathbb{Z}/p\mathbb{Z})$. We denote this kernel by $\Delta_1$. Note that $\Delta_1$ is a $p$-group. Therefore it is nilpotent. Since $\mathrm{PSL}_d(\mathbb{Z}/p\mathbb{Z})$ is simple, this implies that $\mathrm{PSL}_d(\mathbb{Z}/p^n\mathbb{Z})$ has one composition factor equal to $\mathrm{PSL}_d(\mathbb{Z}/p\mathbb{Z})$ and all of its other composition factors are abelian. However, $\mathrm{PSL}_d(\mathbb{Z}/p^n\mathbb{Z})$ has no nontrivial abelian quotients. Therefore $\mathrm{PSL}_d(\mathbb{Z}/p\mathbb{Z})$ is the only simple group quotient of $\mathrm{PSL}_d(\mathbb{Z}/p^n\mathbb{Z})$. However, the Frattini subgroup is the intersection of all maximal normal subgroups (because if $H \le G$ is maximal then $\cap_g gHg^{-1}$ is maximal normal). Therefore, the Frattini subgroup is the intersection of all kernels of homomorphisms to simple groups. This proves the claim.

So $\mathrm{PSL}_d(\mathbb{Z}/q\mathbb{Z})/\Phi(\mathrm{PSL}_d(\mathbb{Z}/q\mathbb{Z}))$ is isomorphic to \begin{equation} \label{eq.proper-1} \mathrm{PSL}_d(\mathbb{Z}/p_1\mathbb{Z}) \times \cdots \times \mathrm{PSL}_d(\mathbb{Z}/p_k\mathbb{Z}). \end{equation} Since $\Phi(\mathrm{PSL}_d(\mathbb{Z}/q\mathbb{Z})) \leq K$ we may assume without loss of generality that $n_j = 1$ for all $j \in \{1,\ldots,k\}$. Thus we can apply Lemma \ref{semisimple} to the inclusion $K \le G$ to finish this lemma. \end{proof}

\begin{lemma} \label{lem.small} Let $d \geq 5$ and let $B$ be a subgroup of the image in $\mathrm{PSL}_d(\mathbb{Z})$ of the copy of $\mathrm{SL}_2(\mathbb{Z})$ in the upper left corner of $\mathrm{SL}_d(\mathbb{Z})$. Let $K$ be a proper finite index subgroup of $\mathrm{PSL}_d(\mathbb{Z})$. Then we have $16|BhK| \leq |\mathrm{PSL}_d(\mathbb{Z})gK|$ for all $g,h \in \mathrm{PSL}_d(\mathbb{Z})$, where the cardinality is taken in $\mathrm{PSL}_d(\mathbb{Z})/K$. \end{lemma}

\begin{proof} Using the congruence subgroup property we see that it suffices to show that if $q \geq 2$ then for any proper subgroup $K$ of $\mathrm{PSL}_d(\mathbb{Z}/q\mathbb{Z})$ the maximal size of a $\mathrm{PSL}_2(\mathbb{Z}/q\mathbb{Z})$-orbit in $\mathrm{PSL}_d(\mathbb{Z}/q\mathbb{Z})/K$ is at most $\frac{1}{16}|\mathrm{PSL}_d(\mathbb{Z}/q\mathbb{Z})/K|$.\\
\\
Using Lemma \ref{lem.proper} we see that there exists a prime factor $p$ of $q$ such that if we write $\pi$ for the projection of $\mathrm{PSL}_d(\mathbb{Z}/q\mathbb{Z})$ onto $\mathrm{PSL}_d(\mathbb{Z}/p\mathbb{Z})$ then $\pi(K)$ is a proper subgroup of $\mathrm{PSL}_d(\mathbb{Z}/p\mathbb{Z})$. The map $\pi$ sends $\mathrm{PSL}_2(\mathbb{Z})$-orbits in $\mathrm{PSL}_d(\mathbb{Z}/q\mathbb{Z})$ to $\mathrm{PSL}_2(\mathbb{Z})$-orbits in $\mathrm{PSL}_d(\mathbb{Z}/p\mathbb{Z})$. Moreover is $m$-to-$1$ for some fixed $m$.  Therefore it suffices to show that if $L$ is a proper subgroup of $\mathrm{PSL}_d(\mathbb{Z}/p\mathbb{Z})$ for some prime $p$ then the maximal size of a $\mathrm{PSL}_2(\mathbb{Z})$-orbit in $\mathrm{PSL}_d(\mathbb{Z}/p\mathbb{Z})/L$ is at most $\frac{1}{16}| \mathrm{PSL}_d(\mathbb{Z}/p\mathbb{Z})|$. \\
\\
The $\mathrm{PSL}_2(\mathbb{Z})$-orbits in $\mathrm{PSL}_d(\mathbb{Z}/p\mathbb{Z})/L$ are the double cosets $\mathrm{PSL}_2(\mathbb{Z}/p\mathbb{Z})xL$ for $x \in \mathrm{PSL}_d(\mathbb{Z}/p\mathbb{Z})$. In \cite{MR0506701} it is shown that the maximal size of a proper subgroup of $\mathrm{PSL}_d(\mathbb{Z}/p\mathbb{Z})$ for a prime $p$ is $(p^d-1)(p-1)^{-1}$. For any $d \in \mathbb{N}$ we have \[ |\mathrm{PSL}_d(\mathbb{Z}/p\mathbb{Z})| =  \frac{1}{\mathrm{gcd}(d,p-1)(p-1)} \prod_{j=0}^{d-1}(p^d-p^j) \] so that in particular \[ |\mathrm{PSL}_2(\mathbb{Z}/p\mathbb{Z})| =  \frac{(p^2-p)(p^2-1)}{\mathrm{gcd}(d,p-1)(p-1)}. \] Therefore if $d \geq 5$ we have \begin{align} \frac{|\mathrm{PSL}_d(\mathbb{Z}/p\mathbb{Z})|}{|\mathrm{PSL}_2(\mathbb{Z}/p\mathbb{Z})| \,|L|} &= \frac{1}{|L|} \frac{(p^d-p)(p^d-1)}{(p^2-p)(p^2-1)} \prod_{j=2}^{d-1} (p^d-p^j) \label{eq.small-1} \\ & \geq \frac{1}{|L|} \prod_{j=2}^{d-1} (p^d-p^j) \label{eq.small-2} \\ & \geq \frac{p-1}{p^d-1} \prod_{j=2}^{d-1} (p^d-p^j) \nonumber \\ & = \frac{(p^d-p^2)(p^d-p^3)}{p^d-1}(p-1) \prod_{j=4}^{d-1} (p^d-p^j) \label{eq.small-3} \\ & \geq (p-1)\prod_{j=4}^{d-1}(p^d-p^j) \label{eq.small-4} \\ & \geq 16 \nonumber \end{align} Here, (\ref{eq.small-2}) follows from (\ref{eq.small-1}) and (\ref{eq.small-4}) follows from (\ref{eq.small-3}) since in each case the factor dropped is at least one. It follows that any double coset $\mathrm{PSL}_2(\mathbb{Z}/p\mathbb{Z})xL$ has size at most $\frac{1}{16}|\mathrm{PSL}_d(\mathbb{Z}/p\mathbb{Z})|$ and so the proof of Lemma \ref{lem.small} is complete. \end{proof}

Theorem \ref{thm} is obtained by applying Theorem \ref{thm:general} to the subgroups $A$ and $B$ constructed in Proposition \ref{prop:ping-pong}. Because $A$ is profinitely dense, it surjects onto every finite quotient. In particular, every $B$-orbit in a finite quotient of $\mathrm{PSL}_d(\mathbb{Z})$ is contained in an $A$-orbit. To define the automorphism $\omega:C \to C$, let $A$ be freely generated by $\{a_1,a_2,a_3,a_4\}$ and $B$ be freely generated by $\{b_1,b_2,b_3,b_4\}$. Then $C$ is freely generated by $\{a_j,b_j\}_{j=1}^4$. So there is a unique order 2 automorphism defined by $\omega(a_j)=b_j$ and $\omega(b_j)=a_j$ for $j \in \{1,\ldots,4\}$. By Lemmas \ref{lem.right} and \ref{lem.small} the subgroups $A$ and $B$ satisfy the other conditions of Theorem \ref{thm:general}.

\bibliographystyle{plain}
\bibliography{/Users/peterburton/Documents/Bibliography/pjburton-bibliography.bib}

Department of Mathematics\\
The University of Texas at Austin\\
\texttt{lpbowen@math.utexas.edu}\\
\texttt{pjburton@math.utexas.edu}\\

\end{document}